\documentclass[12pt,reqno]{amsart}
\usepackage{fullpage,amssymb,amsmath,mathrsfs,enumerate,color}

\newtheorem{thm}{Theorem}
\newtheorem{prop}[thm]{Proposition}
\newtheorem{lemma}[thm]{Lemma}

\theoremstyle{definition}

\newtheorem*{convention}{Convention}

\newcommand\setsep{;\ }

\def\er{\mathbb R}

\def\qe{\mathbb Q}
\def\B{\mathcal B}
\def\en{\mathbb N}

\def \rng {\operatorname{Rng}}
\def \dom {\operatorname{Dom}}

\def \Re {\operatorname{Re}}
\def \Im {\operatorname{Im}}

\begin{document}
\title{On projectional skeletons in Va\v{s}\'ak spaces}
\author{Ond\v{r}ej F.K. Kalenda}
\address{Department of Mathematical Analysis \\
Faculty of Mathematics and Physic\\ Charles University\\
Sokolovsk\'{a} 83, 186 \ 75\\Praha 8, Czech Republic}
\email{kalenda@karlin.mff.cuni.cz}
\thanks{Supported by the Research grant GA \v{C}R P201/12/0290.} 
\begin{abstract} We provide an alternative proof of the theorem saying that any Va\v{s}\'ak (or, weakly countably determined) Banach space admits a full $1$-projectional skeleton. The proof is done with the use of the method of elementary submodels and is comparably simple as the proof given by W.~Kubi\'s (2009) in case of weakly compactly generated spaces.\end{abstract}
\maketitle

\section{Introduction}

Investigation of indexed families of bounded linear projections is
an important tool in the study of nonseparable Banach spaces. One of the first achievements in this area is the famous result of Amir and Lindenstrauss \cite{AL} who proved that any weakly compactly generated Banach space admits a projectional resolution of identity. This was later extended to several larger classes -- Va\v{s}\'ak 
\cite{vasak} and Gul'ko \cite{gulko79} proved the same for weakly countably determined Banach spaces (sometimes called Va\v{s}\'ak spaces), Valdivia \cite{val88,val91} for spaces whose dual unit ball is a Corson compact and even for more general classes. Another result in this direction is due to Fabian and Godefroy \cite{FG} who constructed a projectional resolution of the identity in the dual to any Asplund space.

 A projectional resolution of indentity is a transfinite sequence of norm-one projections satisfying certain properties (for the precise definition see the quoted papers or, for example, \cite[Chapter 6]{fabiankniha}). The main application of such families consists in transferring properties from separable spaces to certain nonseparable ones using transfinite induction 
(see  \cite[Chapter 6]{fabiankniha}).

The original constructions of projectional resolutions of identity are quite technical and involved. There were several attempts to put some common structure to these constructions. One of them is based on a notion of the projectional generator introduced by Orihuela and Valdivia \cite{OV} (for an explanation see also \cite[Chapter 6]{fabiankniha}). It seems that the optimal notion is that of a projectional skeleton introduced by Kubi\'s \cite{kubisSkeleton}. Its advantage is that it says something on the structure of the space itself, while a projectional generator is just a technical tool.

Projectional skeletons can be constructed by a small adjustment of the methods of constructing projectional resolutions or projectional generators. In fact, the existence of a projectional generator implies the existence of a projectional skeleton (cf. \cite[Proposition 7]{kubisSkeleton}) and it is not clear whether the converse is true (cf. \cite[Question 1]{kubisSkeleton}).
Kubi\'s also suggested to construct projectional skeletons using the set-theoretical method of elementary submodels. The key tool is \cite[Theorem 15]{kubisSkeleton} which characterizes the existence of a projectional skeleton in terms of elementary submodels. The advantage of this approach is that the common technical part of the construction is covered by a universal result of set-theoretic nature. Hence the structure of Banach spaces in question is used ``just'' to prove the equivalent condition. Moreover, in case of weakly compactly generated spaces this condition can be verified in a surprisingly simple way, see \cite[Proposition 5]{kubisSkeleton}.
In this paper we provide a comparably simple proof for Va\v{s}\'ak spaces.

\section{Preliminaries}
In this section we recall the basic definitions and some known results related to projectional skeletons and Va\v{s}\'ak spaces. We start by the definition of a projectional skeleton.

Let $X$ be a Banach space. A \emph{projectional skeleton} on $X$ is an indexed system of bounded linear projections $(P_\lambda)_{\lambda\in\Lambda}$ where $\Lambda$ is an up-directed set such that the following 
conditions are satisfied:

 \begin{itemize}
	\item[(i)] $P_\lambda X$ is separable for each $\lambda$,
	\item[(ii)] $P_\lambda P_\mu=P_\mu P_\lambda =P_\lambda$ whenever $\lambda\le\mu$,
	\item[(iii)] if $(\lambda_n)$ is an increasing sequence in $\Lambda$, it has a supremum $\lambda\in\Lambda$ and $P_\lambda [X]=\overline{\bigcup_n P_{\lambda_n}[X]}$,
		\item[(iv)]$X=\bigcup_{\lambda\in\Lambda} P_\lambda [X]$.
 \end{itemize}

\noindent The subspace $D=\bigcup_{\lambda\in\Lambda} P^*_\lambda [X^*]$ is called the \emph{subspace induced by the skeleton}. If $D=X^*$, the projectional skeleton is said to be {\em full}.

Due to \cite[Proposition 9]{kubisSkeleton} we can suppose (up to passing to a closed cofinal subset $\Lambda'\subset\Lambda$) that a skeleton satisfies moreover
\begin{itemize}
 \item[(v)] $\sup_{\lambda\in\Lambda}\|P_\lambda\|<\infty$.
\end{itemize} 
Furhter, by \cite[Lemma 10]{kubisSkeleton}, if a projectional skeleton satisfies the condition (v), the condition (iii) can be strengthened to the following condition:
\begin{itemize}
 \item[(iii')] If $(\lambda_n)$ is an increasing sequence in $\Lambda$, it has a supremum $\lambda\in\Lambda$ and $P_\lambda x=\lim_{n\to\infty} P_{\lambda_n}x$ for each $x\in X$.
\end{itemize} 
In the sequel we will assume that a projectional skeleton is a family satisfying conditions (i),(ii),(iii'),(iv) and (v). If $\|P_\lambda\|=1$ for each $\lambda$, we call the respective skeleton {\em $1$-projectional skeleton}.

There is a close relationship between the constant from the condition (v) and the properties of the induced subspace of the dual. This relationship is described in the following lemma. This lemma can be also proved using \cite[Theorem 15, Lemma 14 and Lemma 4]{kubisSkeleton}. However, up to our knowledge it has not been explicitly formulated and our proof is direct and easy.
Recall that a subspace $D\subset X^*$ is called {\em $r$-norming}, where $r\ge 1$, if for any $x\in X$ we have
$$\left\|x\right\|\le r \sup\{|x^*(x)|\setsep x^*\in D, \|x^*\|\le 1\}= \sup\{|x^*(x)|\setsep x^*\in D, \left\|x^*\right\|\le r\}.$$

\begin{lemma} Let $X$ be a Banach space, $(P_\lambda)_{\lambda\in\Lambda}$ a projectional skeleton on $X$ and $D\subset X^*$ the subspace induced by the skeleton.
\begin{itemize}
	\item[(a)] If $\sup_{\lambda\in\Lambda}\|P_\lambda\|=r\in[1,\infty)$, then $D$ is $r$-norming.
	\item[(b)] If $D$ is $r$-norming for some $r\ge 1$, then there is a closed cofinal $\Lambda'\subset\Lambda$ such that $\|P_\lambda\|\le r$ for each $\lambda\in\Lambda'$. 
\end{itemize}
\end{lemma}
 
\begin{proof} 
(a) Let $x\in X$ be arbitrary. Fix $x^*\in X^*$ such that $\|x^*\|=1$ and $|x^*(x)|=\|x\|$.
By the property (iv) of the skeleton there is $\lambda\in\Lambda$ with $P_\lambda x=x$. Then
$$\|x\|=|x^*(x)|=|x^*(P_\lambda x)|=|P_\lambda^* x^*(x)|.$$
Since $P_\lambda^*x^*\in D$ and $\|P_\lambda^*x^*\|\le\|P_\lambda^*\|\cdot\|x^*\|\le r$, this shows that $D$ is $r$-norming.

(b) Suppose that $D$ is $r$-norming and that the skeleton satisfies the condition (v). Set
$$\Lambda'=\{\lambda\in\Lambda; \|P_\lambda\|\le r\}.$$
By the condition (iii') it is clear that $\Lambda'$ is closed in $\Lambda$. So, it is enough to show that $\Lambda'$ is cofinal. 

To do that, fix any $\lambda_0\in\Lambda$. We will construct by induction sequences $(\lambda_n)$, $(A_n)$ and $(B_n)$ such that the following conditions are satisfied for each $n\in\en$.
\begin{itemize}
	\item $A_n$ is a dense countable subset of $P_{\lambda_{n-1}}[X]$.
	\item $B_n$ is a countable subset of $D\cap rB_{X^*}$ such that for each $x\in A_n$ we have
	$\|x\|= \sup\{|x^*(x)|\setsep x^*\in B_n\}$.
	\item $\lambda_n\in\Lambda$, $\lambda_n\ge\lambda_{n-1}$ and $P_{\lambda_n}^*[X^*]\supset B_n$.
\end{itemize}
It is clear that the construction can be done by induction starting from $\lambda_0$. 
Let $\lambda=\sup_n\lambda_n$. Then $\lambda\in\Lambda'$. Indeed, suppose that $\|P_\lambda\|>r$. Fix $x\in X$ and $\varepsilon>0$ such that $\|x\|=1$ and $\|P_\lambda x\|>(1+\varepsilon)r+\varepsilon$. By the property (iii) of the projectional skeleton there is $n\in\en$ and $y\in A_n$ such that $\|y-P_\lambda x\|<\varepsilon$. In particular, then $\|y\|>(1+\varepsilon)r$, hence there is $y^*\in B_n$ with $|y^*(y)|>(1+\varepsilon)r$. Then
$$r\ge\|y^*\|\cdot\|x\|\ge |y^*(x)|=|P_\lambda^*y^*(x)|=|y^*(P_\lambda x)|
\ge |y^*(y)|-|y^*(y-P_\lambda x)|> (1+\varepsilon)r-r\varepsilon=r,$$
a contradiction.
\end{proof}

We continue by recalling the definition of Va\v{s}\'ak spaces. There are several equivalent definitions which are collected in the following theorem.

\begin{thm} Let $X$ be a Banach space. Then the following assertions are equivalent.
\begin{itemize}
	\item[(1)] There is a sequence $(A_n)$ of weak$^*$ compact subsets of $X^{**}$ such that for any $x\in X$ and $x^{**}\in X^{**}\setminus X$ there is $n\in\en$ with $x\in A_n$ and $x^{**}\notin A_n$.
	\item[(2)] There is a family $F_s$, $s\in\bigcup_{k=1}^\infty \en^k$, of weak$^*$ compact subsets of $X^{**}$ indexed by finite sequences of natural numbers, and a set $\Sigma\subset\en^\en$ of infinite sequences of natural numbers such that
	$$X=\bigcup_{\alpha\in\Sigma}\bigcap_{k=1}^\infty F_{\alpha(1),\alpha(2),\dots,\alpha(k)}.$$
	\item[(3)] There is a separable metric space $\Sigma$ and a set-valued mapping $\varphi:\Sigma\to X$ with the properties:
\begin{itemize}
	\item[(a)] For any $s\in\Sigma$ the value $\varphi(s)$ is a nonempty weakly compact subset of $X$.
	\item[(b)] For any $U\subset X$ weakly open the set $\{s\in\Sigma\setsep \varphi(s)\subset U\}$ is open in $\Sigma$.
	\item[(c)] $\varphi$ is onto $X$, i.e., $\bigcup_{s\in\Sigma}\varphi(s)=X$.
\end{itemize}
\end{itemize}
\end{thm}

This theorem follows easily from \cite[Proposition 7.1.1]{fabiankniha} using \cite[Theorem 7.9]{kechris}. A Banach space $X$ satisfying the equivalent conditions of the previous theorem is said to be {\em Va\v{s}\'ak} or {\em weakly countably determined} (cf. \cite[Definition 7.1.5]{fabiankniha}).

A set-valued mapping satisfying the conditions (a) and (b) from the theorem is said to be {\em upper semi-continuous compact-valued}, shortly {\em usc-K}.

We will need the following easy topological property of Va\v{s}\'ak spaces, see \cite[Lemma 3]{vasak} or \cite[Theorem 7.1.4]{fabiankniha}.

\begin{lemma}\label{L:L} Let $X$ be a Va\v{s}\'ak Banach space. Then $X$ is weakly Lindel\"of, i.e., $(X,w)$ is Lindel\"of.
\end{lemma}

\section{Method of elementary submodels}

In this section we briefly recall some basic facts concerning the method of elementary submodels. This set-theoretical method is useful in various branches of mathematics. In particular, Dow in \cite{dow} illustrated its use in topology, Koszmider in \cite{kos05} used it in functional analysis. Later, inspired by \cite{kos05}, Kubi\'s in \cite{kubisSkeleton} gave a mehod of constructing retractional (resp. projectional) skeleton in certain compact (resp. Banach) spaces using this method. In \cite{cuth-fm} the method has been slightly simplified and specified by C\'uth. Another results concerning retractional or projectional skeletons proved using elementary submodels are given in \cite{cuka-mr} and \cite{BHK2}.
 We briefly recall some basic facts. More details may be found e.g. in \cite{cuth-fm} and \cite{cuka-cejm}.

First, let us recall some definitions. Let $N$ be a fixed set and $\phi$ a formula in the language of $ZFC$. Then the {\em relativization of $\phi$ to $N$} is the formula $\phi^N$ which is obtained from $\phi$ by replacing each quantifier of the form ``$\forall x$'' by ``$\forall x\in N$'' and each quantifier of the form ``$\exists x$'' by ``$\exists x\in N$''.

If $\phi(x_1,\ldots,x_n)$ is a formula with all free variables shown (i.e., a formula whose free variables are exactly $x_1,\ldots,x_n$) then  $\phi$ is said to be {\em absolute for $N$} if
\[
\forall a_1,\ldots,a_n\in N\quad (\phi^N(a_1,\ldots,a_n) \Leftrightarrow \phi(a_1,\ldots,a_n)).
\]

A list of formulas, $\phi_1,\ldots,\phi_n$, is said to be {\em subformula closed} if  every subformula of a formula in the list is also contained in the list.

The method is based mainly on the following theorem (a proof can be found in \cite[Chapter IV, Theorem 7.8]{Kunen}).

\begin{thm}\label{T:countable-model}
Let $\phi_1, \ldots, \phi_n$ be any formulas and $Y$ any set. Then there exists a set $M \supset Y$ such that
$\phi_1, \ldots, \phi_n \text{ are absolute for } M$ and $|M| \leq \max(\omega,|Y|)$.
\end{thm}

Since the set from Theorem~\ref{T:countable-model} will often be used, the following notation is useful.

Let $\phi_1, \ldots, \phi_n$ be any formulas and $Y$ be any countable set.
Let $M \supset X$ be a countable set such that $\phi_1, \ldots, \phi_n$ are absolute for $M$.
Then we say that $M$ is an \emph{elementary model for $\phi_1,\ldots,\phi_n$ and $Y$ containing $X$}.
This is denoted by $M \prec (\phi_1,\ldots,\phi_n; Y)$.

The fact that certain formula is absolute for $M$ will always be used in order to satisfy the assumption of the following lemma from \cite[Lemma 2.3]{cuthRmoutilZeleny}. Using this lemma we can force the model $M$ to contain all the needed objects created (uniquely) from elements of $M$.

\begin{lemma}\label{l:unique-M}
Let $\phi(y,x_1,\ldots,x_n)$ be a formula with all free variables shown and $Y$ be a countable set.
Let $M$ be a fixed set, $M \prec (\phi, \exists y \colon \phi(y,x_1,\ldots,x_n);\; Y)$, and
$a_1,\ldots,a_n \in M$ be such that there exists a set $u$ satisfying
$\phi(u,a_1,\ldots,a_n)$. Then there exists $u \in M$ such that $\phi(u,a_1,\ldots,a_n)$.
\end{lemma}

\begin{proof}Let us give here the proof just for the sake of completeness. Using the absoluteness of the formula $\exists u\colon \phi(u,x_1,\ldots,x_n)$ there exists $u\in M$ satisfying $\phi^M(u,a_1,\ldots,a_n)$.
Using the absoluteness of $\phi$ we get, that for this $u\in M$ the formula $\phi(u,a_1,\ldots,a_n)$ holds.
\end{proof}

We shall also use the following convention.

\begin{convention}
Whenever we say ``\emph{for any suitable model $M$ (the following holds \dots)}''
we mean that  ``\emph{there exists a list of formulas $\phi_1,\ldots,\phi_n$ and a countable set $Y$ such that for every $M \prec (\phi_1,\ldots,\phi_n;Y)$ (the following holds \dots)}''.
\end{convention}

By using this  terminology we lose the information about the formulas $\phi_1,\ldots,\phi_n$ and the set $Y$. However, this is not important in applications.

Let us recall several further results about elementary models which we will need. 

\begin{lemma}\label{l:predp}
There are formulas $\theta_1,\dots,\theta_m$ and a countable set $Y_0$ such that any $M\prec(\theta_1,\ldots,\theta_m;\; Y_0)$ satisfies the following conditions:
\begin{itemize}

	\item[(i)] $\er,\mathbb{C},\qe,\qe+i\qe,\en\in M$ and the operations of the addition and multiplication on $\mathbb{C}$, the functions $z\mapsto \Re z$ and $z\mapsto\Im z$ on $\mathbb{C}$ and the standard order on $\er$ belong to $M$.

	\item[(ii)] If $f\in M$ is a mapping, then $\dom(f)\in M$, $\rng(f)\in M$ and $f[M]\subset M$. Further, for any $A\in M$ we have $f[A]\in M$ as well.

	\item[(iii)] If $A$ is finite, then $A\in M$ if and only if $A\subset M$.
   
  \item[(iv)] If $x_1,\dots,x_n$ are arbitrary, then $x_1,\dots,x_n\in M$ if and only if the ordered $n$-tuple $(x_1,\dots,x_n)$ is an element of $M$.
		
	\item[(v)] If $A\in M$ is a countable set, then $A\subset M$.

	\item[(vi)] If $A,B\in M$, then $A\cup B\in M$, $A\cap B\in M$, $A\setminus B\in M$.

	\item[(vii)] If $A,B\in M$, then $A\times B\in M$.
	
	\item[(viii)] If $X\in M$ is a real vector space, then $X\cap M$ is $\qe$-linear.

	\item[(ix)] If $X\in M$ is a complex vector space, then $X\cap M$ is $(\qe+i\qe)$-linear.

	\item[(x)] If  $X\in M$ is a Banach space, then $X^*\in M$ as well.
	
	\item[(xi)] If $X\in M$ is a separable metric space, then there is a dense countable set $C\subset X$ with $C\in M$ and there is a countable basis $\B$ of $X$ such that $\B\in M$.

\end{itemize}
\end{lemma} 

\begin{proof} The proof is based on Lemma \ref{l:unique-M} and some of these results can be found in \cite{cuth-fm}.
The precised statement is proved in \cite[Lemma 4.3]{BHK2}. More precisely, the quoted lemma contains assertions (i)--(x) and
the existence of the set $C$ in (xi). However, the existence of $\B$ can be proved similarly using the formula
\begin{multline*}\exists\B( \forall B\in\B(B\mbox{ is an open subset of }X)\ \&\ \exists f(f\mbox{ is a mapping of $\en$ onto }\B)\\ \&\ 
\forall G\subset X\forall x\in G(G\mbox{ is open } \Rightarrow \exists B\in \B (x\in B\subset G)))\end{multline*}
\end{proof}

The following lemma is the key tool for constructing projections using elementary submodels. It goes back to \cite[Lemma 4]{kubisSkeleton} and in the current form it is proved in \cite[Lemma 5.1]{BHK2}. The fourth assertion which is not explicitly stated in \cite{BHK2} is an immediate consequence of the first three ones.

\begin{lemma}\label{le:projekce}
For a suitable elementary model $M$ the following holds: Let $X$ be a Banach space and $D\subset X^*$ an $r$-norming subspace.
If $X\in M$ and $D\in M$, then the following hold:

\begin{itemize}
	\item $\overline{X\cap M}$ is a closed linear subspace of $X$;

	\item $\overline{X\cap M}\cap (D\cap M)_\perp=\{0\}$;

	\item the canonical projection of  $\overline{X\cap M}+(D\cap M)_\perp$ onto  $\overline{X\cap M}$ along $(D\cap M)_\perp$
	has norm at most $r$;

  \item $\overline{X\cap M}+(D\cap M)_\perp$ is a closed linear subspace of $X$.
\end{itemize}
\end{lemma} 

Finally, the following lemma characterizes the existence of a projectional skeleton using elementary submodels. It essentially follows from  \cite[Theorem 15]{kubisSkeleton}, the current form is proved in \cite[Lemma 5.2]{BHK2}. If the condition (ii) is fulfilled, following \cite{kubisSkeleton} we say that {\em $D$ generates projections on $X$}.

\begin{lemma}\label{le:generovani} Let $X$ be a Banach space and $D\subset X^*$ a norming subspace. Then the following two assertions are equivalent
\begin{itemize}
	\item[(i)] $X$ admits a projectional skeleton such that $D$ is contained in the subspace induced by the skeleton.
	\item[(ii)] For any suitable elementary model $M$ 
	$$\overline{X\cap M}+(D\cap M)_\perp=X.$$ 
\end{itemize}
\end{lemma}

\section{The result}

Our main result is the following theorem:

\begin{thm} Let $X$ be a (real or complex) Va\v{s}\'ak Banach space. Then $X^*$ generates projections on $X$. In particular, $X$ admits a full $1$-projectional skeleton.
\end{thm}

This result is known as it can be proved by combination of known results. Indeed, if $X$ is Va\v{s}\'ak, then it is weakly Lindel\"of determined by \cite[Theorem 2.5]{mer87} and hence $X^*$ generates projections on $X$ by \cite[Corollary 25]{kubisSkeleton}. However, the proof of the quoted result of \cite{mer87} is quite involved and we present here a simple direct proof. It is simple in the sense that most of the technicalities are hidden in the abstract method of elementary submodels. The proof follows immediately from Proposition~\ref{P:2} below. 

Let us stress that our proof works, without any extra effort, simultaneously for both real and complex spaces.	

To give the proof we will need also the following proposition. It is a generalization of \cite[Proposition 5]{kubisSkeleton} where the statement of our main theorem is proved for weakly compactly generated Banach spaces.

\begin{prop}\label{P:1} For any suitable elementary model $M$ the following holds: Let $X$ be a Banach space and $K\subset X$ a weakly compact subset. Suppose that $X\in M$ and $K\in M$. Then $K\subset \overline{X\cap M}+(X^*\cap M)_\perp$.
\end{prop}

\begin{proof} Let $\phi_1,\dots,\phi_N$ be a subformula-closed list of formulas which contains the formulas from Lemma~\ref{l:predp}, the formulas from Lemma~\ref{le:projekce} and the formulas below marked by $(*)$. Let $Y$ be a countable subset containing the set $Y_0$ from Lemma~\ref{l:predp} and the countable set provided by  Lemma~\ref{le:projekce}. Fix an arbitrary $M\prec(\phi_1,\dots,\phi_N;Y)$. Suppose that $X\in M$ and $K\in M$.

We claim that $K\subset \overline{X\cap M}+(X^*\cap M)_\perp$. Suppose, on the contrary, that $K\not\subset \overline{X\cap M}+(X^*\cap M)_\perp$. Fix some $x\in K\setminus(\overline{X\cap M}+(X^*\cap M)_\perp)$.
Since $\overline{X\cap M}+(X^*\cap M)_\perp$ is closed linear subspace of $X$ due to Lemma~\ref{le:projekce}, by the Hahn-Banach theorem we can find $f\in X^*$ such that $f|_{\overline{X\cap M}+(X^*\cap M)_\perp}=0$ and $\Re f(x)>1$.
Since $f\in ((X^*\cap M)_\perp)^\perp$, the bipolar theorem yields $f\in\overline{X^*\cap M}^{w^*}$ (note that $X^*\cap M$ is $\qe$-linear (or $\qe+i\qe$-linear in the complex case) by Lemma~\ref{l:predp}(viii)-(x)).

Further, since $K\ne\emptyset$ (note that $x\in K$), the absoluteness of the formula
$$\exists y: y\in K\eqno{(*)}$$
shows that $K\cap M\ne\emptyset$.

For any $y\in \overline{K\cap M}^w$ we have $f(y)=0$, hence there is $g_y\in X^*\cap M$ such that $\Re g_y(y)<1$ and $\Re g_y(x)>1$. The set
$$U_y=\{z\in\overline{K\cap M}^w\setsep \Re g_y(z)<1\}$$
is relatively weakly open in $\overline{K\cap M}^w$ and contains $y$. Hence $U_y$, $y\in\overline{K\cap M}^w$ is an open cover of the compact set
$\overline{K\cap M}^w$ (we refer to the weak topology). Hence there are $y_1,\dots,y_n\in \overline{K\cap M}^w$ such that $\overline{K\cap M}^w=U_{y_1}\cup\dots\cup U_{y_n}$.

Set $A=\{g_{y_1},\dots,g_{y_n}\}$. By the above $A\subset M$. Since $A$ is finite, we get $A\in M$ by Lemma~\ref{l:predp}(iii). Note that $\Re g(x)>1$ for each $g\in A$. Therefore by the absoluteness of the formula
$$\exists y\in K\; \forall g\in A: \Re g(y)>1\eqno{(*)}$$
we get $y\in K\cap M$ such that $\Re g(y)>1$ for $g\in A$. But this is a contradiction with the fact that $K\cap M\subset U_{y_1}\cup\dots\cup U_{y_n}$.
\end{proof}

\begin{prop}\label{P:2} For any suitable elementary model $M$ the following holds: Let $X$ be a Banach space, $\Sigma$ a separable metric space and $\varphi:\Sigma\to (X,w)$ a usc-K map which is onto $X$. If $X,\Sigma,\varphi\in M$, then $\overline{X\cap M}+(X^*\cap M)_\perp=X$.
\end{prop}

\begin{proof} Let $\phi_1,\dots,\phi_N$ be a subformula-closed list of formulas which contains the formulas from Lemma~\ref{l:predp}, Lemma~\ref{le:projekce} and Proposition~\ref{P:1} and the formulas below marked by $(*)$. Let $Y$ be a countable subset containing the set $Y_0$ from Lemma~\ref{l:predp} and the countable sets provided by  Lemma~\ref{le:projekce} and Proposition~\ref{P:1}. Fix an arbitrary $M\prec(\phi_1,\dots,\phi_N;Y)$. Assume that $X,\Sigma,\varphi\in M$.

Since $\Sigma\in M$ and $\Sigma$ is a separable metric space, by Lemma~\ref{l:predp}(xi) 
we can fix a countable dense subset $A\subset \Sigma$ and a countable basis $\B$ of $\Sigma$ such that $A\in M$ and $\B\in M$. Note that by Lemma~\ref{l:predp}(v) we have also $A\subset M$ and $\B\subset M$.

For any $x\in A$ its image $\varphi(x)$ is a weakly compact subset of $X$ and $\varphi(x)\in M$ by Lemma~\ref{l:predp}(ii), hence $\varphi(x)\subset \overline{X\cap M}+(X^*\cap M)_\perp$ by Proposition~\ref{P:1}. Moreover, since $\varphi$ is usc-K and $\overline{X\cap M}+(X^*\cap M)_\perp$ is weakly closed (being a closed linear subspace by Lemma~\ref{le:projekce}), the set $$\{x\in \Sigma\setsep
\varphi(x)\cap (\overline{X\cap M}+(X^*\cap M)_\perp)\ne\emptyset\}$$ is closed in $\Sigma$. Hence we get
$$
\begin{aligned}
\forall x\in A&:\varphi(x)\subset \overline{X\cap M}+(X^*\cap M)_\perp,\\
\forall x\in \Sigma&:\varphi(x)\cap (\overline{X\cap M}+(X^*\cap M)_\perp)\ne\emptyset.
\end{aligned}$$

Suppose that $\overline{X\cap M}+(X^*\cap M)_\perp\ne X$. By Hahn-Banach theorem we get $f\in X^*$ such that $f\ne 0$ but $f|_{\overline{X\cap M}+(X^*\cap M)_\perp}=0$. By the bipolar theorem we get $f\in\overline{X^*\cap M}^{w^*}$. Fix $z\in X$ with $\Re f(z)>1$ and some $x\in \Sigma$ with $z\in\varphi(x)$.

For any $y\in\overline{X\cap M}+(X^*\cap M)_\perp$ we can find $g_y\in X^*\cap M$ with $\Re g_y(y)<\frac14$ and $\Re g_y(z)>1$. Then 
$$U_y=\left\{u\in X\setsep \Re g_y(u)<\frac14\right\}$$
is a weakly open set. Moreover, $U_y$, $y\in\overline{X\cap M}+(X^*\cap M)_\perp$ is a weakly open cover of $\overline{X\cap M}+(X^*\cap M)_\perp$.
This set, being a weakly closed subset of $X$ is weakly Lindel\"of (by Lemma~\ref{L:L}). Hence we can find a sequence $(y_n)$ such that $\bigcup_{n=1}^\infty U_{y_n}\supset\overline{X\cap M}+(X^*\cap M)_\perp$.

Enumerate $\{B\in\mathcal B\setsep x\in B\}=(B_n)_{n=1}^\infty$
and set $C_n=B_1\cap\dots\cap B_n$. Since $B_n\in M$ for each $n$, by Lemma~\ref{l:predp}(vi) we have also $C_n\in M$. Further, set $E_n=\{g_{y_1},\dots,g_{y_n}\}$. By Lemma~\ref{l:predp}(iii) we get $E_n\in M$.

Fix $n\in\en$. The choice $a=x$ and $u=z$ witnesses that the formula
$$\exists a\in C_n\;\exists u\in\varphi(a)\;\forall g\in E_n: \Re g(u)>1\eqno{(*)}$$
is satisfied. Using elementarity we find $a_n\in C_n\cap M$ and $u_n\in\varphi(a_n)\cap M$ such that $\Re g(u_n)>1$ for all $g\in E_n$.

Since $a_n\in C_n$ for each $n\in\en$, we get $a_n\to x$. Since $\varphi$ is usc-K, by \cite[Lemma 3.1.1]{fabiankniha} the sequence $(u_n)$ has a weak cluster point $u\in\varphi(x)$, hence $u\in\overline{\varphi(x)\cap M}^{w}\subset\overline{X\cap M}$.

Further, $\Re g_{y_j}(u_n)>1$ whenever $n\ge j$, therefore $\Re g_{y_j}(u)\ge1$ for all $j\in \en$. On the other hand, $u$ must be covered by some $U_{y_j}$, 
hence $\Re g_{y_j}(u)<\frac14$ for some $j\in\en$. It is a contradiction completing the proof.
\end{proof}

\def\cprime{$'$}


\end{document}